\documentclass[smallextended]{svjour3}     
\smartqed
\usepackage{graphicx}
\usepackage{mathptmx}
\usepackage{amsmath}
\usepackage{latexsym}
\usepackage{amssymb}
\usepackage{xy}

\newcommand{\Z}{{\mathbb Z}}

\newcommand{\C}{{\mathbb C}}
\newcommand{\R}{{\mathbb R}}

\newcommand{\QQ}{{\mathcal Q}}

\newcommand{\UU}{{\mathcal U}}

\newcommand{\oooo}{\overline}

\newcommand{\nnn}{\nabla}

\DeclareMathOperator{\id}{id}

\DeclareMathOperator{\rank}{rank}

\let\<\langle
\let\>\rangle
\DeclareMathOperator{\Id}{Id}
\DeclareMathOperator{\diag}{diag}
\let\tilde\widetilde
\let\hat\widehat
\newcommand{\MR}[1]{}

\begin{document}

\title{Integrable Harmonic Higgs Bundles With Vanishing $\mathcal{U}$ And Eigenvalues of $\mathcal{Q}$}
\author{Jiezhu Lin, Xuanming Ye}
\institute{Jiezhu Lin \at
            Department of Mathematics and Information Science ,Guangzhou University,Guangzhou, P. R. China \\
              \email{ljzsailing@163.com}\\
          Xuanming Ye \at
            Department of Mathematics and Information Science ,Guangzhou University,Guangzhou, P. R. China \\
              \email{youwentizhaolaoye@163.com}
}

\maketitle
\begin{abstract}
We study the tt*-geometry with vanishing endormorphism $\UU$. Given an integrable harmonic Higgs bundle $(E, h, \Phi, \mathcal{U},\mathcal{Q})$ on a complex manifold $M$, Firstly  we prove that, under the \emph{IS} condition, vanishing $\UU$ implies vanishing Higgs field $\Phi$ and the Chern connection of the Hermitian Einstein metric $h$ is a holomorphic connection, so the metric $h$ and $\mathcal{Q}$ are invariant. Secondly, without the \emph{IS} condition, we show that vanishing $\UU$ will imply vanishing Higgs field $\Phi$ if we assume that the Chern connection of $h$ is a holomorphic connection. Finally, we add real structure $\kappa$. Given any \emph{CV}-structure, we prove that super-symmetric operator $\mathcal{Q}$ must have $0$ as an eigenvalue when the underlying bundle has odd rank.
\end{abstract}
\section{Introduction}\label{intro}

Cecotti and Vafa \cite{CV3} \cite{CVN} considered moduli spaces of
$N=2$ super-symmetric quantum field theories and introduced a geometry on them which is governed by the tt*-equations.
Tt*-structure is understood well after the work of C. Hertling in \cite{Hert2}, as an enrichment that of harmonic bundle previously introduced by N. Hitchin and C. Simpson. Tt*-structure was axiomatized as a \emph{CV}-structures by C. Hertling in \cite{Hert2}.
One way to study tt*-structure is to construct the so-called \emph{TERP}-structure and prove it to be tr\emph{TERP}-structure. The exsitence of a tt*-structure on the base space of semiuniversal unifolding of hypersurface singularity was proved by C. Hertling, by using oscillating integrals, and he proved that this structure was compatible with the Frobenius structure and get a \emph{CDV}-structure. The existence of a \emph{CDV}-structure on the base space of a convenient and non-degenerate Laurent polynomial was proved by C. Sabbah in \cite{Sabb22}.
Another way to build tt*-structure on deformation space of Landau-Ginzburg model was developed by H. Fan in \cite{Fan}. By considering the spectrum theory of twisted Lapacian operator, he obtained tt*-structure on the deformation space.
Most recent work on $\Hat{\Gamma}$-integral structure in orbifold quantum cohomology has been done by A. Chiodo, H. Iritani and Y. Ruan. in \cite{CIR}.

An integrable harmonic Higgs bundle $(E, h,\Phi, \mathcal{U},\mathcal{Q})$ is a harmonic bundle $(E,h,\Phi)$ on a complex manifold $M$ with supplementary structures $\mathcal{U}$ and $\mathcal{Q}$, here $\mathcal{U}$ and $\mathcal{Q}$ are endomorphisms of complex vector bundle $H$ associated to $E$. Adding a real structure $\kappa$ on a integrable harmonic Higgs bundle $(E, h,\Phi, \mathcal{U},\mathcal{Q})$ in a compatible way we can get a \emph{CV}-structure.
When the \emph{CV}-structure is semi-simple everywhere, the associated \emph{TERP}(w)-structure is determined  completely by the number $w$, its Stokes matrix $S$ and eigenvalues of $\mathcal{U}|_{t_0}$. However, the moduli spaces of massive deformations of conformal field theories contain non semi-simple points. The points $t_0$ with $\mathcal{U}|_{t_0}=0$ correspond to conformal field theories. The eigenvalues of $\mathcal{Q}|_{t_0}$ at such points are charges, certain rational number $\alpha_j$ such that $(-1)^we^{-2\pi i \alpha_j}$ are the eigenvalues of the monodromy. In the singularity case they are up to a shift the spectral numbers(=the exponents-1) \cite{Hert2}.

The purpose of this article is to study the integrable harmonic Higgs bundles $(E, h, \Phi, \mathcal{U},\mathcal{Q})$ with vanishing $\mathcal{U}$, and to study on eigenvalues of $\mathcal{Q}$ for any \emph{CV}-structure. We firstly conclude that, given an integrable harmonic Higgs bundle, vanishing $\mathcal{U}$ will implies vanishing of the Higgs field $\Phi$  if $\mathcal{Q}$ saitisfies the \emph{IS condition} at $p$, here "\emph{IS condition}" means that differences of any two eigenvalues of $\mathcal{Q}|_{p}$ are not $\pm 1$. Moreover, the $(1,0)$ part of  Chern connection $D'$ of the Hermitian Einstein $h$  is a holomorphic connection. The structure connection $\tilde{\nabla}$ have regular singularity at $0$,  the pullback bundle $p^*E$ is the logarithmic lattice,  and $\tilde{\nabla}$ can be decomposition as a direct sum of $r=\rank E$  meromorphic connection with regular singularities at $0$, here $p:\mathbb{C}^{*}\times M \rightarrow M$ is the projection.  For a \emph{CDV}-structure, since $\mathcal{U}=-\Phi_{\mathcal{E}}$, Obviously $\Phi=0$ implies $\mathcal{U}=0$. However the inverse is usually not true for a \emph{CV}-structure. It is quite interesting that the inverse is true when the $\mathcal{Q}$ satisfies the \emph{IS} condition at one point.

Secondly, we consider the case that some difference of two eigenvalues of $\mathcal{Q}$ may be equal to $1$. In this case we conclude that an integrable harmonic Higgs bundle $(\mathcal{T}_M, h, \Phi, \mathcal{U},\mathcal{Q})$, $\mathcal{U}=0$ will implies that $\Phi=0$ with the assumption that the $(1,0)$ part of  Chern connection $D'$ is equal to zero.

Finally, given any \emph{CV}-structure, we prove that super-symmetric operator $\mathcal{Q}$ must have $0$ as an eigenvalue when the underlying bundle $E$ has odd rank. we give the results on eigenvalues of $\mathcal{Q}$ for a \emph{CV}-structure $(\mathcal{T}_M, h, \Phi, \kappa,\mathcal{U},\mathcal{Q})$ when $\dim M$ is equal to $2$ and $3$.
\begin{acknowledgements}
We would like to thanks Claude Sabbah for valuable comments, and for pointing out a mistake in Lemma \ref{Lemma2} in the first version of the paper.
\end{acknowledgements}

\section{Frobenius manifold and tt* geometry}\label{section1}
In this section we recall the notion of a Frobenius manifold, integrable harmonic Higgs bundle and tt*-bundle.
 This will mainly serve to fix notation.
\subsection{Saito structure and Frobenius manifold structure}\label{subsection1a}
Frobenius manifolds were introduced and investigated by B.
Dubrovin as the axiomatization of a part of the rich mathematical
structure of the Topological Field Theory (TFT): cf. \cite{D},
\cite{Hert}, \cite{Mani}.

A Frobenius manifold (also called Frobenius structure on $M$) is a
quintuple $(M, \circ, g, e, \mathcal{E})$. Here $M$ is a manifold
in one of the standard categories ($C^\infty$, analytic, ...), $g$
is a metric on $M$ (that is, a symmetric, non-degenerate bilinear
form, also denoted by $\< \,,\, \>$), $\circ$ is a commutative and
associative product on the tangent bundle $\mathcal{T}_M$ and
depends smoothly on $M$, such that if $\nabla$ denotes the
Levi-Civita connection of $g$ and $\Theta_M$ denotes the locally
free sheaf of $\mathcal{O}_M$-module corresponding to
$\mathcal{T}_M$, all subject to the following conditions:
\begin{enumerate}
\item[a)] $\nabla$ is flat;

\item[b)] $g( X \circ Y, Z) = g( X, Y \circ Z)$, for any $X, Y, Z
\in \Theta_M$.

\item[c)] the unit vector field e is covariant constant w.r.t.
$\nabla$
\begin{eqnarray*}
\nabla e = 0;
\end{eqnarray*}

\item[d)] Let
\begin{equation*}
c(X, Y, Z):=g(X \circ Y, Z)
\end{equation*}
(a symmetric 3-tensor). We require the 4-tensor
\begin{equation*}
(\nabla_Z c)(U, V, W)
\end{equation*}
to be symmetric in the four vector fields $U, V, W, Z$.

\item[e)] A vector field $\mathcal{E}$ must be determined on $M$
such that
\begin{align*}
\nabla (\nabla \mathcal{E}) &= 0;\\
\mathcal{L}_{\mathcal{E}}(\circ)&=\circ;\\
\exists d \in \mathbb{C}, \quad \mathcal{L}_{\mathcal{E}} ( g)&= (2-d) \cdot g.
\end{align*}
\end{enumerate}

Locally, given a Frobenius manifold structure on an open subset $U
\subset \mathbb{C}^{m}$, Let $t=(t^1, t^2,\dots, t^m)$ be
holomorphic local coordinates of $U$ such that $e=\partial_{t^1}$,
then we can find a function $F=F(t)$ such that its third
derivatives
$$C_{ijk}:=\frac{\partial F}{\partial t^i \partial t^j \partial
t^k}$$ satisfy the following equations
\begin{enumerate}
\item[1)] Normalization:
$$g_{ij}:=C_{1ij}$$ is a constant non-degenerate
matrix. Let
$$(g^{ij}):=(g_{ij})^{-1}$$

\item[2)] Associativity: the functions
$${C_{ij}}^k:= \sum_{l} C_{ijl}
\cdot g^{lk}$$ define a commutative and associative algebra on
$T_t M$ by
$$\partial_{t^i} \circ \partial_{t^j}:= \sum_k C_{ij}^k
\partial_{t^k}$$

\item[3)] Homogeneity: The function $F$ must be quasi-homogeneous,
i.e.,
\[
\mathcal{L}_{\mathcal{E}} F = d_F \cdot F + \text{quadratic
terms},
\]
where $\mathcal{E}=\sum_{i,j} (q_i^j t^i + r^j )\partial_{t^j}$,
and $d_F \in \mathbb{C}$.
\end{enumerate}

If the endomorphism $\nabla \mathcal{E}$ is semi-simple, then the
Euler vector field can be reduced to the form
$$\mathcal{E} = \sum_i d_i t^i \partial_{t^i} + \sum_{j \mid d_j =0} r_j \partial_{t^j}.$$
where all $r_j$ are complex numbers, and all $d_i$ are the
eigenvalues of $\nabla \mathcal{E}$. Moreover, if $g(e, e)=0$, we
have
\begin{proposition}[\cite{D}]\label{proposition4}
Let $M$ be a Frobenius manifold. Assume that $g(e, e)=0$ and that
the endomorphism $\nabla \mathcal{E}$ is semi-simple. Then by a
linear change of coordinates t$_i$ the matrix $g_{ij}$ can be
reduced to the anti-diagonal form
\begin{align*}
g_{ij}=\delta_{i+j,m+1};\\
e=\partial_{t^1}
\end{align*}
and in these coordinates, write
\begin{align*}
F(t)= \frac{1}{2} (t^1)^2 t^m + \frac{1}{2} t^1 \sum_{i\geq 2} t^i
t^{m-i+1} + f(t^2, t^3,\dots, t^m)
\end{align*}
for some functions, the sum
$$d_i + d_{m-i+1}$$ does not depend on $i$, and
$$d_F = 2d_1 + d_m.$$
If the degrees are normalized in such a way that $d_1 =1$ then
they can be represented in the form
$$d_i = 1-q_i;\quad d_F=3-d,$$
where $q_1, q_2,\dots, q_m,d$ satisfy
$$q_1 = 0,\quad q_m = d,\quad q_i
+q_{m-i+1} =d.$$
\end{proposition}
So, under the assumption of Proposition \ref{proposition4}, we can
choose flat holomorphic local coordinates $t^1, t^2,\dots, t^m$ of
$M$ such that $g_{ij}=\delta_{i+j,m+1}$, $e=
\partial_{t^1}$ and
\begin{align*}
\mathcal{E}=\sum_{i}d_{i}t^{i}\partial_{t^i}+ \sum_{i |d_i =0}
r^{i} \partial_{t^i};\\
d_1 =1;\\
d_i + d_{m+1-i}=2-d;
\end{align*}
\subsection{Harmonic Higgs bundles with supplementary structures}\label{subsec:supplstruct}
In this paragraph, we consider supplementary structures on a harmonic Higgs bundle. 
Let $M$ be a complex manifold and Let $E$ be holomorphic bundle on $M$, equipped with a Hermitian non-degenerate sequilinear
form $h$. We will say that $\left(E, h\right)$ is a Hermitian holomorphic bundle. For any operator $P$ acting on $E$ , we will
denote by $P^\dag$ its adjoint with respect to $h$. A holomorphic Higgs field $\Phi$ on $E$ we means an $\mathcal{O}_M$-linear morphism
$\Phi:E \longrightarrow \Omega^1_M \otimes E$ satisfying the integrability relation $\Phi \wedge \Phi=0$, we  then say that $\left( E, \Phi\right)$
is a Higgs bundle.

Let $(E,h)$ be a Hermitian holomorphic bundle with Higgs field $\Phi$. Let $H$ be the associated $C^\infty$ bundle, so that $E=Ker \overline \partial$, let $D=D^{'}+\overline{\partial}$ be the Chern connection of $h$ and let $\Phi^\dag$ be the $h$-adjoint of $\Phi$. We say that $(E,h,\Phi)$ is a harmonic Higgs bundle $($or that $h$ is \emph{Hermite-Einstein} with respect to $(E,\Phi)$ $)$ if $D^{'}+\overline{\partial}+\Phi+\Phi^\dag$ is an integrable connection on~$H$. This is equivalent to a set of relations:
\begin{align}
(\overline{\partial})^2=0,\overline{\partial}(\Phi)=0, \Phi \wedge \Phi=0; \\
(D')^2=0, D^{'}(\Phi^\dag)=0,\Phi^\dag \wedge \Phi^\dag=0; \\
 D'(\Phi)=0, \overline{\partial}(\Phi^\dag)=0, D'\overline{\partial}+\overline{\partial}D'=-(\Phi \Phi^\dag+\Phi^\dag \Phi).
\end{align}
where the first line is by definition, the second one by $h$-adjunction from the first one, and the third line contains the remaining relations in the integrability condition of $D^{'}+\overline{\partial}+\Phi+\Phi^\dag$.
\begin{definition}[\cite{Sabbah2}]\label{realHB}
Let $(E, h)$ be a Hermitian holomorphic bundle, Let $H$ be the $C^\infty$ bundle. By a \emph{real structure} we mean an antilinear isomorphism $\kappa:H \widetilde{\rightarrow} H $ such that
\begin{align}
\kappa^2=\Id; \\
h(\kappa \cdot, \kappa \cdot)=\overline{h(\cdot, \cdot)}; \\
D(\kappa)=0.
\end{align}
\end{definition}

\begin{remark}
Set
$g(X,Y)=h(X,\kappa Y),$ we get a  nondegenerate bilinear form $g$. Obviously $g$ is symmetric and compatible with $D$,i.e., $D(g)=0.$
\end{remark}

\begin{definition}[\cite{Sabbah2}]\label{potentialHHB}
Let $(E, h, \Phi)$ be a harmonic Higgs bundle, if there exists a $C^\infty$ endomorphism $A$ of $E$ satisfying $\Phi=D^{'}A.$ we will say that $(E, h, \Phi)$ is a \emph{potential harmonic Higgs bundle}, also denoted by $(E, h, A)$
\end{definition}

\begin{definition}[\cite{Sabbah2}]\label{integrableHHB}
An \emph{integrable harmonic Higgs bundle} is a tuple $(E, h, \Phi, \mathcal{U},\mathcal{Q} )$,  here $(E, h,\Phi)$ is a harmonic Higgs bundle, and there exist two endomorphisms $\mathcal{U}$ and  $\mathcal{Q}$ of $H$ satisfying
\begin{align}
\overline{\partial}(\mathcal{U})=0; \label{U}\\
\mathcal{Q}^\dag= \mathcal{Q}. \label{QQ} \\
[\Phi, \mathcal{U}]=0; \label{CU}\\
D'(\mathcal{U})- [\Phi, \mathcal{Q}]+\Phi=0; \label{UCQ}\\
D'(\mathcal{Q})+ [\Phi,\mathcal{U}^\dag]=0; \label{QCU}
\end{align}
\end{definition}

\begin{remark}
Given any harmonic harmonic Higgs bundle, if we set
$$\tilde{\nabla}=D^{'}+\overline{\partial}+\frac{1}{z}\Phi+ z \Phi^{\dag} + (\frac{\mathcal{U}}{z}-\mathcal{Q}-z \mathcal{U}^\dag)\frac{dz}{z}$$
It is an integrable connection on the pull-back bundle $\pi: p^*E \rightarrow \mathbb{C}^{*}\times M$. The $(0,1)$-part of the connection $\overline{\partial}+z \Phi^{\dag}$ gives a holomorphic structure on pullback bundle. and $(1,0)$-part of this connection
$D^{'}+\frac{1}{z}\Phi+ (\frac{\mathcal{U}}{z}-\mathcal{Q}-z \mathcal{U}^\dag)\frac{dz}{z}$ is called the \emph{structure connection} of the integrable  harmonic Higgs bundle.
\end{remark}
Putting all the structure together, we get
\begin{definition}[\cite{Sabb22}]\label{tt*}
 A \emph{tt*-bundle} is a tuple $(E, h,\Phi,\kappa, \mathcal{U},\mathcal{Q})$, such that $(E, h,\kappa)$ be a real Hermitian holomorphic bundle,
$(E, h,\Phi, \mathcal{U},\mathcal{Q})$ is an integrable harmonic Higgs bundle, and moreover,
 \begin{align*}
\mathcal{U}^*=\mathcal{U},\\
\mathcal{Q}^*+\mathcal{Q}=0.
\end{align*}
\end{definition}

\begin{remark}\label{CVHert}
A tt*-bundle is a \emph{CV}-structure, and an integrable harmonic Higgs bundle is \emph{CV}-structure without real strucure $\kappa$ in \cite{Hert2}.
\end{remark}

\begin{definition}[\cite{Sabbah2}]\label{hFM}
A structure of \emph{harmonic Frobenius manifold} $(M, \circ, g,\kappa,e,\mathcal{E})$ on a complex manifold such that
$(M, \circ, g,e,\mathcal{E})$ is a Frobenius manifold, and $(M,h,\kappa, \Phi,\mathcal{U},\mathcal{Q})$ is a real integrable harmonic Higgs bundle,
with supplementary condition $D_e e=0$ and $d\in \mathbb{R}.$
Here
$\Phi_X Y:=-X\circ Y ,h(X,Y):=g(X, \kappa Y),\mathcal{U}:=-\Phi_{\mathcal{E}} ,\mathcal{Q}:=D^{'}_{\mathcal{E}}-\mathcal{L}_{\mathcal{E}}-\frac{2-d}{2} \Id.$
\end{definition}
\begin{remark}[\cite{Sabbah2}]\label{CDV}
A structure of harmonic Frobenius manifold is a manifold with a \emph{CDV}-structure in \cite{Hert2}.
\end{remark}

\begin{proposition}[\cite{Sabbah2}]\label{1}
There is a canonical harmonic structure on the canonical Frobenius manifold attached to a convenient and nondegenerate Laurent polynomial. The corresponding Hermitian metric $h$ is positive definite.
\end{proposition}
The existence of  tt*-structure of rank two was completely discussed in \cite{taka}.
The existence of a canonical harmonic structure(\emph{CDV}-structre) on base space of a semi-universal unfolding of a hypersurface singularity was prove in \cite{Hert2}. The existence of a canonical harmonic structure on base space of a universal unfolding of a convenient and non-degenerate Laurent polynomails was proved in \cite{Sabb22} .  A suffucient and necessary condition for a Frobenius manifold to be a harmonic Frobenius manifold was given by the first author in \cite{Lin}, and she construct a real structure $\kappa$ on a Frobenius manifold to be harmonic Frobenius manifold with vanishing $\mathcal{Q}$. The integral structure called $\hat{\Gamma}$-integral structure on quantum D-modules was done in \cite{CIR}. More recent work on tt*-structure on Landau-Ginzburg side has been done in \cite{FLY}.
\subsection{Correspondence with special integrable harmonic Higgs bundles}\label{s3.2}
Harmonic bundles was introduced by Simpson to a generalization of variations of polarized Hodge structure.
But from harmonic bundle one can not recover the Hodge filtration, integrable harmonic Higgs bundle provides such information.
\begin{example}[\cite{Sabbah2}]\label{Sabb1}
(Variations of complex Hodge structures of weight 0)
Let $H$ be a $C^\infty$ vector bundle on $M$, equipped with a flat connection $\tilde{\nabla}=\tilde{\nabla}^{'}+\tilde{\nabla}^{''}$ and a composition $H=\oplus_{p\in \mathbb{Z}}H^p$ into $C^\infty$ subbundles. We assume that Griffiths transversality relations hold:
$$\tilde{\nabla}^{'} H^p \subset ( H^p \oplus H^{p-1})\otimes_{\mathcal{O}_M} \Omega_M^1,\tilde{\nabla}^{''} H^p \subset ( H^p \oplus H^{p+1})\otimes_{\mathcal{O}_{\overline M}} \Omega_{\overline{M}}^1$$
We denoted by $D_{|H^p}$ the composition of $\tilde{\nabla}_{|H^p}$ with the projection to $H^p,$ denoted by $\Phi_{|H^p}$ the composition of $\tilde{\nabla}^{'}_{|H^p}$ with the projection to $H^{p-1},$
and by $\Phi^\dag_{|H^p}$ that of $\tilde{\nabla}^{''}_{|H^p}$ with the projection to $H^{p+1},$ then we set
 $$D=\oplus_p D_{|H^p}, \Phi= \oplus_p \Phi_{|H^p},\Phi^\dag= \oplus_p \Phi^\dag_{|H^p}$$

Assume that we are given a non-degenerate Hermitian form $k$ such that $\tilde{\nabla}(k)=0$ and the decomposition $H=\oplus_{p\in \mathbb{Z}}H^p$ is $k$-orthogonal. Consider the nondegenerate Hermitian form $h=\oplus_p (-1)^p k_{|H^p}$. Then $D(h)=0$ and $\Phi^\dag$ is complex Hodge structure of weight $0$. In particular, $(H, D^{''}, h, \Phi)$ is a harmonic Higgs bundle. Set $\mathcal{Q}=\oplus_p p \Id_{|H^p}$ and $\mathcal{U}=0.$ we get $D(\mathcal{Q})=0$ and as $p$ is real, we have $\mathcal{Q}^\dag=\mathcal{Q}.$ Lastly, we have $[\Phi, \mathcal{Q}]=\Phi.$
By a real structure $\kappa$, we mean an anti-linear involution $\kappa: H \rightarrow H$ which is $\tilde{\nabla }$-horizontal such that  $\kappa (H^p)= H^{-p}$ for any $p$. Then $D(\kappa)=0$ and $\Phi^\dag=\kappa \Phi \kappa.$ The previous data thus define a tt*-bundle.
\end{example}

The inverse of example \ref{Sabb1} is straightforward. We formulate it.
\begin{lemma}[\cite{Hert2}]\label{t3.4}
Let $(H\to M,D,\Phi,\kappa,h,\UU,\QQ)$ be a tt*-bundle with
$\UU=0$ and such that $\QQ$ has no eigenvalues in
$\frac{w+1}{2}+\Z$.

Define a connection $\nnn:=D+\Phi+\Phi^\dag$ and define
\begin{eqnarray*} \label{3.29}
H^{p,w-p}_t&:=&\bigoplus_{\alpha:\ [\alpha+\frac{w+1}{2}]=p}
\ker(\QQ-\alpha\id:H_t\to H_t)\ ,\\
F^p_t&:=&\bigoplus_{q\geq p}H^{q,w-q}_t\ ,\\
S&:&H_t\times H_t \to \C \mbox{ \ \ with}\\
S(a,b)&:=& (2\pi i)^w(-1)^ph(a,\oooo b) \mbox{ \ \ for }
a\in H^{p,w-p}_t,\ b\in H_t\ ,\nonumber\\
A &:=& e^{2\pi i\QQ}\ .
\end{eqnarray*}
Then $(H\to M,\nnn,H_\R,S,F^\bullet,A)$ is a variation of polarized
Hodge structures of weight $w$ with an automorphism $A$.
\end{lemma}

The eigenvalues of $\mathcal{Q}$ gives the decomposition of the Hodge decomposition. We are interested in the explicit computation on
eigenvalues of the matrix the $\mathcal{Q}$, we shall see these eigenvalues determined the Higgs field locally.

\bigskip

\section{Main Result}\label{}

In this paper, we study the integrable harmonic Higgs bundle $\left( E, h,\Phi, \mathcal{U}, \mathcal{Q} \right)$.  The Hermitian metric $h$ will be always assumed to be positive-definite. The Chern connection of $h$ is denoted by $D'+\overline{\partial}.$
Firstly, we assume that the differences of any two eigenvalues of $\mathcal{Q}$ is neither $1$ nor $-1.$

\begin{definition}\label{IS}
Let $E$ be a complex vector bundle on $M$, and $\mathcal{Q}$ is an endomorphism of $E$, given any point $p\in M$, we say that $\mathcal{Q}$ satisfies the \emph{IS} condition at $p$ if any difference of two eigenvalues of $\mathbb{Q}(p)$ is not $\pm 1$, here $\mathbb{Q}(p)$ is the matrix of $\mathcal{Q}(p)$ under some local frame.
\end{definition}
Under this \emph{IS} condition ,we prove that  $D^{'} \left( \mathcal{U} \right)=0$ imply $\Phi=0$ locally. Under this conditon, we can conclude that the $(1,0)$-part of the Chern connection $D'$ is a holomorphic connection, and we can choose a holomorphic $D'$-flat local frame
$e_1,e_2,\cdots,e_m$ such that the matrix $\mathbb{Q}$ under the local frame $e_1,e_2,\cdots,e_m$ is constant and diagonal.

\begin{theorem}\label{U0}
Let $\left( E, h,\Phi, \mathcal{U}, \mathcal{Q} \right)$ be an integrable harmonic Higgs bundle.  Assume $\exists p\in M,$ such that $\mathcal{Q}$ satisfies the \emph{IS} condition at $p$, i.e., the differences of any two eigenvalues of $\mathcal{Q} ( p) :E_p\rightarrow E_p$ are  not $\pm 1$, then
\begin{enumerate}
\item[(1) ] There exists an open neighborhood $U_p$ of $p$,  such that $\mathcal{Q}$ satisfies \emph{IS} condition at all $q\in U_p$.
\item[(2) ] Locally, $\Phi$ is uniquely determined by $D^{'}\left( \mathcal{U}\right).$
If $D^{'}\left( \mathcal{U}\right)=0$ holds, then $\Phi=0$.  In this case, the connection $D^{'}$ is holomorphic,i.e. $D^{'}\overline{\partial}+\overline{\partial}D^{'}=0.$  If $\mathcal{Q}$ satisfies \emph{IS} condition at all $q\in M$. then
$\left( E, h,\Phi, \mathcal{U}, \mathcal{Q} \right)$ is a potential integrable harmonic Higgs bundle with  a potential $\mathcal{U}$.
\item[(3) ] There exists a flat holomorphic local frame $\left\{ e_1,e_2,...,e_m \right\} \subset \Gamma \left( M,\mathcal{T} _{M} \right) $,
    the matrix of $\mathcal{Q}$ is a diagonal matrix  $\Lambda =\mathrm{diag}\left( \lambda _1,\lambda _2,...,\lambda _m \right) , \lambda _j\in \mathbb{R}.$ Moreover, specially, if $\mathcal{U}=0$ , then
    Locally, the structure connection can be written in a simple way  $$\tilde{\nabla}
   = D^{'}+\frac{-\Lambda +\frac{w}{2}I_m}{z} dz$$ i.e.,$\tilde{\nabla}$ can be written as direct sum of $m$ line bundles with connections ${\tilde{\nabla}}^i=d+\frac{\lambda_i + \frac{w}{2} }{z}$.
\end{enumerate}
\end{theorem}

\begin{corollary}\label{U1}
Let $\left( \mathcal{T} _{M}, h,\Phi, \mathcal{U} =0, \mathcal{Q} \right)
$ be a integrable harmonic Higgs bundle on $M$, set $X \circ Y=-\Phi_X Y$ for arbitrary $X, Y\in \Theta_M$.
Assume $\exists p\in M,$ such that $\mathcal{Q}$ satisfies the \emph{IS} condition at $p$,
then $(M, \circ)$ is a locally trivial pre-Frobenius manfold, i.e. $X \circ Y=0,\forall X, Y\in \Gamma(U_p, \mathcal{T}_M)$.
\end{corollary}
Secondly, we will consider the case without \emph{IS} condition. We restrict to the case that $E=\mathcal{T}_M$.  Under the assumption that $(1,0)$-part of the Chern connection $D$ is holomorphic,  we can also conclude that $D^{'} \left( \mathcal{U} \right)=0$ implies the Higgs field $\Phi=0$ locally.

\begin{theorem}\label{withpm1}
Let $\left( \mathcal{T} _{M}, h, \Phi, \mathcal{U}, \mathcal{Q} \right)$
 be an integrable harmonic Higgs bundle on complex manifold $M$ with $D'(\mathcal{U})=0$, here $D'+ \overline{\partial}$
is the Chern connection of positive-definite Hermitian metric $h$. Assume that $D'$ is holomorphic, Set $X\circ Y:=-\Phi_X Y$, then
\begin{enumerate}
\item[(1).] $X\circ Y=0, \forall X, Y\in \mathcal{T} _{M};$

\item[(2).] If $\mathcal{U}=0$ holds, the structure connection can be written in a simple way  $$\tilde{\nabla}
   = D^{'}+\frac{-\Lambda +\frac{w}{2}I_m}{z} dz$$
\end{enumerate}
\end{theorem}

In \cite{SK2},\cite{SK3}, M. Saito studied the Gauss-Manin connection of hypersurface singularities and developed the notation of the primitive forms. His work was completed by M. Saito \cite{SM2} and resulted in a construction of Frobenius manifolds. A partial Fourier transform maps the Gauss-Manin connection to a \emph{TERP}(w)-structure.  The \emph{TERP}(n+1)-structure ,constructed on the base space $M$ of a semiuniversal unfolding $F$ of a singularity $f:(\mathbb{C}^{n+1},0)\rightarrow (\mathbb{C},0)$ was shown to be generically a tr\emph{TERP}(n+1)-structure by C. Hertling. The \emph{CV}-structure constructed in this way is compatible with the Frobenius manifold structure and gives a \emph{CDV}-structure in \cite{Hert2}. Hertling gave the following conjecture
\begin{conjecture}
Given any $p\in M$, The set $R$ does not contain the $\mathcal{E}+\overline{\mathcal{E}}$ orbit of $p$. If one goes far enough along the flow $\mathcal{E}+\overline{\mathcal{E}}$, then one will not meet anymore the set $R$, the Hermtian metric $h$ will be positive definite, and the eigenvalues of $\mathcal{Q}$ will be tend to $Exp(F_p)-\frac{n+1}{2}$. Here $R$ is the set where the \emph{TERP}(n+1)-structure is not a tr\emph{TERP}(n+1)-structure;$Exp(F_p) :=\bigcup_{x \in \mathrm{Sing} \left (F_p \right)}Exp(F_p,x)$.
\end{conjecture}
He prove that the conjecture is true when $\mathcal{U}_p$ is either have $\mu$ different eigenvalues or $\mathcal{U}_p$ is nilpotent. Here $\mu$ is the Milnor number of $f$, i.e. $\mu$ is the dimension of the Jacobi algebra $\mathcal{O}_{\mathbb{C}^{n+1},0} /(\frac{\partial f}{\partial x_0},\frac{\partial f}{\partial x_1},\cdots, \frac{\partial f}{\partial x_n})$.

\bigskip
In the last part of paper, we study a general tt*-bunle(i.e., a \emph{CV}-structure).  That is an integrable harmonic Higgs bundle with a compatible real structure $\kappa$. We ask for neither $\mathcal{Q}=D^{'}_{\mathcal{E}}-\mathcal{L}_{\mathcal{E}}-\frac{2-d}{2}\Id$ nor $\mathcal{U}=-\Phi_{\mathcal{E}}$.
Given a tt*-bundle $(E, h, \Phi, \kappa, \mathcal{U},\mathcal{Q})$, then compatible conditions include the relation
$\mathcal{Q}=-\kappa \mathcal{Q}\kappa.$ Since $\kappa ^2=\Id,$ by straightforward computation we conclude that the matrices $-\mathbb{Q}$ and $\overline{\mathbb{Q}}$ have the same eigenvalue polynomial, Here $\mathbb{Q}$ is the matrix of $\mathcal{Q}$ under a local frame of $E$.  Now let us  fix a point $p$ in $M$.
If the Hermitian Einstein metric $h$ is positive-definite, all eigenvaluse of $\mathcal{Q}(p)$ are real numbers. Hence together with the condition
$\mathcal{Q}=-\kappa \mathcal{Q}\kappa$, we can conclude that when  $\rank E$ is $2l$, then the matrix of $\mathcal{Q}(p)$ can be diagonalized to $$\diag (\lambda_1, \lambda_2, \cdots, \lambda_l,-\lambda_1, -\lambda_2, \cdots, -\lambda_l ).$$  when $\rank E$ is $2l+1$, $0$ must be an eigenvalue of $\mathbb{Q}(p)$, and the matrix of $\mathcal{Q}(p)$ can be diagonalized to $$\diag (\lambda_1, \lambda_2, \cdots, \lambda_l, 0, -\lambda_1, -\lambda_2, \cdots, -\lambda_l ),$$ all $\lambda_j$ are non-negative real numbers.
\begin{proposition}\label{prop1}
Let $\left( E, h,\Phi, \mathcal{U}, \mathcal{Q}, \kappa \right)
$ be a tt*-bundle on $M$. $r=\rank E$, fixing any $p\in M,$ then
\begin{enumerate}
\item[(1).] If $r=2l+1$, $0$ must be an eigenvalue of $\mathcal{Q}(p)$, and there exist $l=[r/2]$ non-negative real numbers $\lambda_1, \lambda_2, \cdots, \lambda_l$ such that he matrix of $\mathcal{Q}(p)$ can be diagonalized $$\diag(\lambda_1, \lambda_2, \cdots, \lambda_l, 0,  -\lambda_1, -\lambda_2, \cdots, -\lambda_l);$$
\item[(2).] If $r=2l$, there exist $l=r/2$ non-negative real numbers $\lambda_1, \lambda_2, \cdots, \lambda_l$ such that
the matrix of $\mathcal{Q}(p)$ can be diagonalized either to the matrix $$\diag(\lambda_1, \lambda_2, \cdots, \lambda_l, -\lambda_1, -\lambda_2, \cdots, -\lambda_l).$$
\end{enumerate}
\end{proposition}
\begin{remark}
For any point $p\in M,$ the trace of the matrix $\mathbb{Q}(p)$ is equal to zero.
\end{remark}
If $\rank E=2,$ then the eigenvalues of $\mathbb{Q}$ should be $\{ \lambda,-\lambda\}$. If $\rank E=3,$ then the eigenvalues of $\mathbb{Q}$ should be $\{ \lambda, 0, -\lambda\}$.
we restricts to the cases that $\dim M=3$ and $\dim M=2.$
\begin{corollary}\label{rank3}
Let $(\mathcal{T}_M, h, \Phi, \kappa, \mathcal{U},\mathcal{Q})$ be a tt*-bundle on $M$ with $D'(\mathcal{U})=0,$ here $\dim M=3,$  $\forall p\in M,$ \\
$(1)$ If neither $\pm 1$ nor $\pm \frac{1}{2}$ is eigenvalues of $\mathbb{Q}(p)$, then there is open neighborhood $U_p$ of $p$ such that $\Phi_{|U_p}=0$ and the connection  $D'_{|U_p}$ is holomorphic.\\
$(2)$ If either $\pm 1$ or $\pm \frac{1}{2}$ is an eigenvalue of $\mathbb{Q}(p)$, and if $D'$ is a holomorphic connection, then there is a flat holomorphic local frame $X_1, X_2, X_3$ such that the matrix $\mathbb{Q}$ satisifying
\begin{equation*}
\begin{aligned}
\mathbb{Q}=\left( \begin{matrix}
	-1&		0&     0\\
    0&      0&     0\\
	0&		0&     1\\
\end{matrix} \right)
\end{aligned}
\end{equation*}
or
\begin{equation*}
\begin{aligned}
\mathbb{Q}=\left( \begin{matrix}
	-\frac{1}{2}&		0&     0\\
    0&                  0&     0\\
	0&		            0&     \frac{1}{2}\\
\end{matrix} \right)
\end{aligned}
\end{equation*}
and locally $\Phi=0$ holds.\\
$(2)$ For the case
\begin{equation*}
\begin{aligned}
\mathbb{Q}=\left( \begin{matrix}
	-1&		0&     0\\
    0&      0&     0\\
	0&		0&     1\\
\end{matrix} \right)
\end{aligned}
\end{equation*}
The monodromy representation $T$ of the local system determined by $(p^{*}\mathcal{T}_M^{(1,0)}, \tilde{D})|_{\{p\}\times \mathbb{C}^*}$ is unity matrix of size $3$.
\end{corollary}
when $\dim M=2$, we get more explicit results.
\begin{corollary}\label{rank2}
Let $(\mathcal{T}_M, h, \Phi, \kappa, \mathcal{U},\mathcal{Q})$ be a tt*-bundle on $M$ with $D'(\mathcal{U})=0,$ here $\dim M=2,$ Given any point $p\in M,$ \\
$(1)$ If $\pm \frac{1}{2}$ is NOT an eigenvalue of $\mathbb{Q}(p)$, then there is open neighborhood $U_p$ of $p$ such that $\Phi_{U_p}=0$ and the connection  $D'_{|U_p}$ is holomorphic.\\
$(2)$ If $\frac{1}{2}$ is an eigenvalue of $\mathbb{Q}(p)$, and if $D'$ is a holomorphic connection, then there is a flat holomorphic local frame $X_1, X_2$ such that the matrix $\mathbb{Q}$ satisifying
\begin{equation*}
\begin{aligned}
\mathbb{Q}=\left( \begin{matrix}
	\frac{1}{2}&     0\\

	0&		    -\frac{1}{2}\\
\end{matrix} \right)
\end{aligned}
\end{equation*}
and locally $\Phi=0$ holds.
\end{corollary}
\bigskip
\section{PROOF OF THE THEOREMS}
In order to prove theorem \ref{U0}, we need some Lemmas in the following.
\begin{lemma}\label{Lemma1}
Let $( E, h,\Phi, \mathcal{U}, \mathcal{Q})$ be an integrable harmonic Higgs bundle on $M$. Given any point $p\in M$, if $\mathcal{Q}$ satisfies \emph{IS} condition at $p\in M$, then there exists an open neighborhood $U_{p}\subset M$ of $p$ such that $\forall q\in U_{p}$, $\mathcal{Q}$ satisfies \emph{IS} condition at $q\in U_{p}$
\end{lemma}
\begin{proof}
Let $\alpha$ and $\beta$ are the eigenvalue functions of $\mathcal{Q}$ over a open neighborhood $V$ of $p_0$, which means
there exists $ X,Y\in \Gamma ( V,E ) $ such that
$$\mathcal{Q} \left( X \right) =\alpha  X
$$
$$\mathcal{Q} \left( Y \right) =\beta Y
$$
Note that $\alpha -\beta \in C^{\infty}\left( V \right) $, if $\left( \alpha -\beta \right) \left( p \right) \notin \left\{ \pm 1 \right\}$, then $\,\,\left( \alpha -\beta \right) \left( p \right) \in \mathbb{R} \backslash\left\{ \pm 1 \right\}
$
Therefore, there exists an open neighborhood $U_{p}$ of $p$, such that
$$( \alpha -\beta ) ( U_{p} ) \in \mathbb{R} \backslash\left\{ \pm 1 \right\}.
$$
$$\therefore \forall q\in U_{p},   ( \alpha -\beta) ( q ) \neq \pm 1.
.$$
\hfill $\qed$
\end{proof}

\begin{lemma}\label{Lemma2}
Let $( E, h,\Phi, \mathcal{U}, \mathcal{Q})$ be an integrable harmonic Higgs bundle on $M$.
 If $D^{'}\overline{\partial}+\overline{\partial}D^{'}=0$ and $D^{'}(\mathcal{Q})=0,$
then $\forall p\in M, $ there exists an open neighborhood $U_p$ of $p$ such that all eigenvalue functions $\alpha \in C^{\infty} \left( U_p \right) \,\,$ of $\mathcal{Q}$ are constants.
\\
\end{lemma}
\begin{proof}
Since the connection $D^{'}$ is holomorphic and $D^{'}+\overline{\partial}$ is compatible with $h$, we can choose a holomorphic $D^{'}$-flat local frame $S_1, S_2, \cdots,S_r$ such that
$$h(S_\alpha, S_\beta)=\delta_{\alpha \beta}$$
here $r=\rank E$.
 By condition $\mathcal{Q}^\dag=\mathcal{Q}$ we can conclude that $$\mathbb{Q}^t=\overline{\mathbb{Q}}.$$ Here $\mathbb{Q}$ is the matrix
 of $\mathcal{Q}$ under the local frame $S_1, S_2, \cdots,S_r$.
 $D^{'}\left( \mathcal{Q} \right)=0$ implies that all entries of the matrix $\mathbb{Q}$ are anti-holomorphic functions, together with $\mathbb{Q}^t=\overline{\mathbb{Q}}$, we conclude that $\mathbb{Q}$ is a constant matrix. So we can choose another holomorphic
 $D^{'}$-flat local frame $e_1, e_2, \cdots,e_r$, such that the matrix $\mathbb{Q}$ is equal to a constant diagonal matrix $\Lambda=\diag(\lambda_1,\lambda_2,\cdots,\lambda_r),$ here all $\lambda_j$ are constants.
\hfill $\qed$
\end{proof}
\begin{lemma}[Lemma 2.16,\cite{Sabbah2}]\label{Lemma3}
Let $A\in M_k(\mathbb{C})$ and $B\in M_l(\mathbb{C})$ be two matrices, then the following properties are equivalent:
\begin{enumerate}
\item[(1) ] For any $Y$ of size $l\times k$ with entries in $\mathbb{C},$ there exists a unique matrix $X$ of the same kind satisfying $XA-BX=Y;$
\item[(2) ] the square matrices $A$ and $B$ have no common eigenvalue.
\end{enumerate}
\end{lemma}

\begin{proof}
{\it of theorem \ref{U0}.}
By Lemma\ref{Lemma1}, we conclude that there exists an open neighborhood $U_p$ of $p$ such that the difference of any two eigenvalues of $\mathcal{Q}$ is neither $1$ nor $-1$. The first statement holds obviously.

Since the connection $D^{'}$ is flat , we can choose a local frame $S_1,S_2, \cdots,S_r$ of $E$ such that $D^{'}S_i=0, \forall i.$
Denoted by $\mathbb{U}$  the matrix of the endomorphism $\mathcal{U}$, denoted by $\mathbb{Q}$ the matrix of the endomorphism $\mathcal{Q}$, denoted by  $C_{(i)}$ the matrix of the endomorphism  $-\Phi_{\partial_i}$, under the local frame $S_1,S_2, \cdots,S_r$, here $t^1,t^2, \cdots,t^m$ are any holomorphic local coordinates of $M,$  $r=\rank E$ and $m=\dim M.$
Denoted $\frac{\partial}{\partial t^i}$ by $\partial_i$ for simplicity.
Since $D^{'}\left( \mathcal{U}\right)-\left[ \Phi ,\mathcal{Q} \right] +\Phi =0$, by straightforward computation ,we conclude that
$$\partial_i \mathbb{U}= C_{(i)}\mathbb{Q}-\left(\mathbb{Q}-I_r \right)C_{(i)}.$$
Here $I_r=\diag (1,1,\cdots,1)$.
By the assumption that the differences of any two eigenvalues is not $\pm 1$, then matrices $\mathbb{Q}$ and $\mathbb{Q}-I_r $ have no common eigenvalues, so by Lemma \ref{Lemma3}, $\forall Y, \exists| X,$ such that $X \mathbb{Q}-\left(\mathbb{Q}-I_m \right)X = Y.$  Taking $Y=\partial_i \mathbb{U}, \exists| C_{(i)}$ such that $\partial_i \mathbb{U}= C_{(i)}\mathbb{Q}-\left(\mathbb{Q}-1 \right)C_{(i)}$ holds. If $\partial_i \mathbb{U}=0,$ then all $C_{(i)}$ must be zero. So we conclude that locally $\Phi=0.$

Obviously $\Phi^\dag =0$ since $\Phi=0$,  Hence $D^{'}\bar{\partial}+\bar{\partial}D^{'}=-\left( \Phi \land \Phi ^{\dagger}+\Phi ^{\dagger}\land \Phi \right)$ implies that
$$D^{'}\bar{\partial}+\bar{\partial}D^{'}=0,$$ i.e., $D^{'}$ is a holomorphic connection.
By the condition $D^{'}(\mathcal{Q})+[\Phi, \mathcal{U}^\dag]=0,$ we get  $D^{'}(\mathcal{Q})=0,$ hence by Lemma\ref{Lemma2},
we can choose a $D^{'}$-flat holomorphic local frame $e_1,e_2, \cdots,e_r$ of $E$
such that
$$\mathbb{Q}=\Lambda =\diag( \lambda _1,\lambda _2,...,\lambda _r ), \forall \lambda _i \in \mathbb{R}.$$

Finally, if $\mathcal{U}=0,$ Obviously we get $\mathcal{U}^\dag=0$ and $D^{'}(\mathcal{U})=0,$
by above discussion, we can choose a holomorphic $D^{'}$-flat local frame $e_1,e_2, \cdots e_r$ such that the matrix of $\mathcal{Q}$
is a diagonal constant matrix $\Lambda.$
Locally, the structure connection,  $$\widetilde{D}=D^{'}+\frac{1}{z}\Phi +\frac{-\mathcal{Q} +\frac{\omega}{2}I_r}{z}dz
= D^{'}+\frac{-\Lambda +\frac{w}{2}I_r}{z} dz $$ $\widetilde{D}$ can be written as direct sum of $r$ holomorphic line bundles $\mathcal{L}_i$ with connections $d+\frac{\lambda_i + \frac{w}{2} }{z}$.
\hfill $\qed$
\end{proof}
\bigskip
We assume that the connection $D'$ is holomprhic, then $D'(\mathcal{U})=0$ implies that $\Phi=0$. For giving a proof of theorem\ref{withpm1}, we need some lemmas.
\begin{lemma}\label{pm1}
Let $R$ be a ring, $V$ be a free $R$-module of finite rank $m$, $\circ$ is a commutative and associative product on $V$. Suppose we have a decomposition of submodules
$V=\oplus_{j=1}^s V_{\lambda+j}$, satisfying
\begin{enumerate}
\item[($1^\circ$)] $X\circ Y\in V_{\lambda+j},\forall X, Y\in V_{\lambda+j+1}, , \forall j=1,2,\cdots,s ;$

\item[($2^\circ$)] $ X\circ Y=0,\forall X\in V_{\lambda+i}, Y\in V_{\lambda+j}, \forall i \neq j;$

\item[($3^\circ$)] $X\circ Y=0,\forall X, Y\in V_{\lambda+1} .$
\end{enumerate}
Assume that for any base $e_1,e_2,\cdots, e_m$ of $V$, we have
$C_{(j)}\overline{C_{(j)}^t}=\overline{C_{(j)}^t}C_{(j)}$, here $C_{(j)}$ is the matrix of the $R$-modules morphism $e_j \circ:V \rightarrow V$ under the base $e_1,e_2,\cdots, e_m$,
then
\begin{equation}\label{zero}
X \circ Y=0, \forall X, Y \in V=\oplus_{j=1}^s V_{\lambda+j}.
\end{equation}
\end{lemma}

\begin{proof}
We prove (\ref{zero}) by induction on $s\in \mathbb{N}.$\\
When $s=1,$ it is trivial because of the assumption $(3^\circ)$.\\
When $s=2,$ then $V=V_{\lambda+1}\oplus V_{\lambda+2}$, here $V_{\lambda+1}$ and $V_{\lambda+2}$ are the submodules of $V$. \\
Suppose $V_{\lambda+1}$ is generated by $e_1,e_2,\cdots, e_k$,  and $V_{\lambda+2}$ is generated by $e_{k+1},e_{k+2},\cdots, e_m$, by the assumption condition $(2^\circ),(3^\circ)$, we have
$$C_{(i)}=0, i\in \{1,2,\cdots, k\}.$$
By the assumption, we have $e_\beta \circ e_i=0, \forall i\in \{1,2,\cdots, k\},\beta \in \{k+1,k+2,\cdots,m\}$\\
$\forall  \beta \in \{k+1,k+2,\cdots,m\}$, by the assumption condition $(1^\circ)$, we can set\\
$e_{\beta} \circ e_{k+1}=f_{11}^{\beta} e_1 +f_{12}^{\beta} e_2+\cdots f_{1k}^{\beta} e_k$,\\
$e_{\beta} \circ e_{k+2}=f_{21}^{\beta} e_1 +f_{22}^{\beta} e_2+\cdots f_{2k}^{\beta} e_k$,\\
$\cdots,\cdots,\quad \cdots,\cdots $\\
$e_{\beta} \circ e_{m}=f_{m-k,1}^{\beta} e_1 +f_{m-k,2}^{\beta} e_2+\cdots f_{m-k,k}^{\beta} e_k$,\\
Denoted by $A^\beta=(f^\beta_{ij})_{(m-k)\times k}$, then we get,
\begin{equation*}
\begin{aligned}
C_{\left( \beta \right)}=\left( \begin{matrix}
	O&		O\\
	A^\beta &		O\\
\end{matrix} \right)
\end{aligned}
\end{equation*}
so we get
\begin{equation*}
\begin{aligned}
\overline{C_{\left( \beta \right)}}^t=\left( \begin{matrix}
	O&		\overline{A^\beta}^t \\
	O &		O\\
\end{matrix} \right)
\end{aligned}
\end{equation*}
since $C_{(j)}\overline{C_{(j)}^t}=\overline{C_{(j)}^t}C_{(j)}$, by straightforward computation ,we conclude
$$A^\beta \cdot \overline{{A^\beta}^t}=0, \forall \beta \in \{k+1,k+2,\cdots,m\},$$ hence $A^\beta=0, \forall \beta \in \{k+1,k+2,\cdots,m\}$.
that is $C_{(j)}=0,\forall j=1,2,\cdots,m.$ Hence we have thus prove (\ref{zero}) in the case $s=2.$\\
Suppose (\ref{zero}) holds when $V=\oplus_{j=1}^s V_{\lambda+j}$, we shall prove that (\ref{zero}) holds when $V=\oplus_{j=1}^{s+1} V_{\lambda+j}$
Suppose that
$e_1,e_2,\cdots, e_k$ is a base of $\oplus_{j=1}^s V_{\lambda+j}$, and $e_{k+1},e_{k+2},\cdots, e_m $ is a base of $V_{\lambda+s+1}$,
By the assumption $(2^\circ)$, we have $$e_\beta \circ e_i=0, \forall i\in \{1,2,\cdots, k\},\beta \in \{k+1,k+2,\cdots,m\}.$$
$\forall  \beta \in \{k+1,k+2,\cdots,m\},$ we can set\\
$e_{\beta} \circ e_{k+1}=f_{11}^{\beta} e_1 +f_{12}^{\beta} e_2+\cdots f_{1k}^{\beta} e_k$,\\
$e_{\beta} \circ e_{k+2}=f_{21}^{\beta} e_1 +f_{22}^{\beta} e_2+\cdots f_{2k}^{\beta} e_k$,\\
$\cdots,\cdots,\quad \cdots,\cdots $\\
$e_{\beta} \circ e_{m}=f_{m-k,1}^{\beta} e_1 +f_{m-k,2}^{\beta} e_2+\cdots f_{m-k,k}^{\beta} e_k$,\\

Denoted by $A^\beta=(f^\beta_{ij})_{(m-k)\times k}$,which is a matrix of size $(m-k)\times k$, then we get
\begin{equation*}
\begin{aligned}
C_{\left( \beta \right)}=\left( \begin{matrix}
	O&		O\\
	A^\beta &		O\\
\end{matrix} \right)
\end{aligned}
\end{equation*}
so we get
\begin{equation*}
\begin{aligned}
\overline{C_{\left( \beta \right)}}^t=\left( \begin{matrix}
	O&		\overline{A^\beta}^t \\
	O &		O\\
\end{matrix} \right)
\end{aligned}
\end{equation*}
since $C_{(\beta)}\overline{C_{(\beta)}^t}=\overline{C_{(\beta)}^t}C_{(\beta)}$, by straightforward computation, we get $A^\beta\overline{A^\beta}^t=0_{(m-k)\times(m-k)}.$ Hence $A^\beta=0_{(m-k)\times k}.$ we conclude
\begin{equation}\label{zero2}
C_{(\beta)}=0,\forall \beta=k+1,k+2,\cdots,s.
\end{equation}
Since we can restrict the endomorphism $e_i \circ$ to the submodule $\oplus_{j=1}^s V_{\lambda+j}$ and get
${e_i \circ }_{|_{\oplus_{j=1}^s V_{\lambda+j}}}:\oplus_{j=1}^s V_{\lambda+j} \longrightarrow \oplus_{j=1}^s V_{\lambda+j}$
Denoted by $B_{(i)}$ the matrix of ${e_i \circ }_{|_{\oplus_{j=1}^s V_{\lambda+j}}}$ under the base $e_1,e_2, \cdots,e_k,$
By straightforward computation, we get
\begin{equation*}
\begin{aligned}
C_{\left( i \right)}=\left( \begin{matrix}
	B_{\left( i \right)}&		O\\
	O&		O\\
\end{matrix} \right)
\end{aligned}
\end{equation*}
by $C_{(i)}\overline{C_{(i)}^t}-\overline{C_{(i)}^t}C_{(i)}=0$ we get $$B_{(i)}\overline{B_{(i)}^t}-\overline{B_{(i)}^t}B_{(i)}=0.$$
So,  by the induction hypothesis, we get $X \circ Y=0, \forall X, Y \in \oplus_{j=1}^s V_{\lambda+j},$ i.e.,  $B_{(i)}=0,\forall i\in {1,2,\cdots, k}.$ Hence we get
\begin{equation}\label{zero3}
C_{(i)}=0,\forall i\in \{1,2,\cdots, k\}.
\end{equation}
By (\ref{zero2})and (\ref{zero3}), we conclude that $X \circ Y=0, \forall X, Y \in V$ holds when $V=\oplus_{j=1}^{s+1} V_{\lambda+j}.$
\hfill $\qed$
\end{proof}

\begin{lemma}
Let $\left( \mathcal{T} _{M}, h, \Phi, \mathcal{U}, \mathcal{V} \right)$
 be an integrable harmonic Higgs bundle on complex manifold $M$ with $D'(\mathcal{U})=0$, here $D'+ \overline{\partial}$
is the Chern connection of positive-definite Hermitian metric $h$. Assume that $D'$ is holomorphic, Set $X\circ Y=-\Phi_X Y$,
 $\forall X, Y, Z\in \mathcal{T}_M $ satisfying $\mathcal{Q}X=\lambda X,\mathcal{Q}Y=\mu Y,$ $\forall Z\in \mathcal{T}_M $
\begin{enumerate}
\item[(1).] If $\lambda \neq \mu $ then $X\circ Y=0$
\item[(2).] If $\lambda = \mu $ then $\mathcal{Q}(X\circ Y )=(\lambda-1)X\circ Y. $
\item[(3).] $\left( X\circ Y \right) \circ Z=0.$
\end{enumerate}
\end{lemma}
\begin{proof}

In fact, $-\left[ \Phi ,\mathcal{Q} \right] +\Phi =0\Longleftrightarrow \left[ -\Phi ,\mathcal{Q} \right] =-\Phi,
$
$$\therefore \,\,\left[ -\Phi _X,\mathcal{Q} \right] \left( Y \right) =\left( -\Phi _X \right) \left( Y \right).
$$
By straightforward computation we get
$$\mathcal{Q} \left( -\Phi _XY \right) =\left( \mu -1 \right) \left( -\Phi _XY \right) .$$
\begin{align}\label{t31}
{\rm i.e.} \quad \mathcal{Q} \left( X\circ Y \right) =\left( \mu -1 \right) X\circ Y
\end{align}
Similarly,
$\left[ -\Phi _Y,\mathcal{Q} \right] \left( X \right) =\left( -\Phi _Y \right) \left( X \right) $ implies
\begin{align}\label{t32}
 \mathcal{Q} \left( Y\circ X \right) =\left( \lambda -1 \right) Y\circ X.
\end{align}
Then by (\ref{t31}),(\ref{t32}) and the assumption that $\Phi$ is symmetric, we have
$$X\circ Y=0.$$

\begin{claim}\label{Tc2}  $\forall X, Y,Z\in \Gamma \left( U,\mathcal{T} _{M}^{1,0} \right) $, if $\mathcal{Q} X=\lambda X, \mathcal{Q} Y=\lambda Y, \mathcal{Q} Z=\lambda Z$, then $ X\circ Y\circ Z=0$.
\end{claim}
In fact, by $\left[ -\Phi ,\mathcal{Q} \right] =-\Phi $ we have
$$
\left[ -\Phi _{X\circ Y},\mathcal{Q} \right] \left( Z \right) =-\Phi _{X\circ Y}Z
$$
\begin{align}\label{t33}
\Longleftrightarrow \mathcal{Q} \left( \left( X\circ Y \right) \circ Z \right) =\left( \lambda -1 \right) \left( \left( X\circ Y \right) \circ Z \right)
\end{align}
If $X\circ Y=0$,then $ \left( X\circ Y \right) \circ Z=0.$
If $X\circ Y\ne 0  $, then $$\mathcal{Q} \left( X\circ Y \right) =\left( \lambda -1 \right) \left( X\circ Y \right). $$
$$ \therefore
\left[ -\Phi _Z,\mathcal{Q} \right] \left( X\circ Y \right) =\left( -\Phi _Z \right) \left( X\circ Y \right)
$$
\begin{align}\label{t34}
\Longleftrightarrow \mathcal{Q} \left( Z\circ \left( X\circ Y \right) \right) =\left( \lambda -2 \right) \left( Z\circ \left( X\circ Y \right) \right).
\end{align}
Then by (\ref{t33}),(\ref{t34}) and the assumption that $\Phi$ is symmetric, we have
$$\left( X\circ Y \right) \circ Z=Z\circ \left( X\circ Y \right) =0.$$
We have $\forall X,Y,Z\in \mathcal{T} _{M}^{1,0},$ if $X,Y,Z$
are eigenvectors, then we have
$$\left( X\circ Y \right) \circ Z=0.$$
$$
\therefore \,\,X\circ \left( Y\circ Z \right) =\left( Y\circ Z \right) \circ X=0=\left( X\circ Y \right) \circ Z.
$$
Therefore
$$ \forall \xi ,\eta ,\zeta \in \Gamma \left( U,\mathcal{T} _{M}^{1,0} \right) , \xi \circ \eta \circ \zeta =0,
$$
especially $\xi^{\circ 3}=0$.
\hfill $\qed$
\end{proof}
\bigskip

\begin{proof}
{\it of theorem \ref{withpm1}.}

Since the connection $D'$ is holomorphic and flat, we can choose a flat holomorphic local frame $S_1,S_2, \cdots, S_m$ such that
$h_{\alpha \beta}:=h(S_\alpha, S_\beta)=\delta_{\alpha \beta}.$ Then the matrix of $\mathcal{Q}$, denoted by $\mathbb{Q}$, is a constant matrix. Since $h_{\alpha \beta}=\delta_{\alpha \beta}$, we get $\mathbb{Q}^t=\overline{\mathbb{Q}}$. Suppose $\lambda_1, \lambda_2,\cdots,\lambda_m$ are the eigenvalues of $\mathbb{Q}$, Then we conclude that all eigenvalues are real numbers. we can assume that  $\lambda_i\leq \lambda_{i+1} \forall i.$ Let $e_1, e_2, \cdots, e_m\in \mathcal{T}_{M}^f$ be the linearly independent  eigenvectors of $\mathbb{Q}$ corresponding $\lambda_1, \lambda_2, \cdots,\lambda_m$.  \\
$Case$ $1^\circ$ (Special Case) If the differnces of any two eigenvalues are not $\pm 1.$  all the conclusion $(1)$ holds by Therorem $2$;\\
$Case$ $2^\circ$ (Special Case) If the set of  all the eigenvalues of $\mathbb{Q}$ are $\{\lambda+1, \lambda+2, \cdots, \lambda+s\},$ we shall conclude $X\circ Y=0, \forall X, Y \in \mathcal{T}_M^{(1,0)}$ in the following.\\
Denoted $V_{\lambda+j}$ by the $\mathcal{O}_M$-module generated by the eigenvectors of $\mathcal{Q}$ corresponding to the eigenvalue $\lambda+j$, set $V=\oplus_{j=1}^s V_{\lambda+j}.$\\ Obviously, $V=\mathcal{T}_M.$
Since $D'(\mathcal{U})=0$, we get the equality $[\Phi, \mathcal{Q}]=\Phi$, which implies that $$X \circ Y=0,\forall X\in V_{\lambda+i},\forall Y\in V_{\lambda+j},i\neq j$$ and
$$X \circ Y\in V_{\lambda+i},\forall X, Y\in V_{\lambda+i+1}.$$
Since  $D'$ is holomorphic, we conclude $-(\Phi\wedge \Phi^\dag +\Phi^\dag \wedge \Phi)=D'\overline{\partial}+\overline{\partial}D'=0$.
Hence we get $C_{(j)}\overline{C_{(j)}^t}=\overline{C_{(j)}^t}C_{(j)}$.
By Lemma \ref{pm1}, we have
\begin{equation}\label{zero4}
X\circ Y=0, \forall X, Y\in V= \oplus_{j=1}^s V_{\lambda+j}.
\end{equation}
$Case$ $3^\circ$ Otherwise, differences of two eigenvalues may be $\pm 1,$ and $\exists \lambda_{j_0} \in\{\lambda_2, \lambda_3, \cdots, \lambda_m\},$ such that $\lambda_{j_0}- 1$ is not an eigenvalue of $\mathcal{Q}$.\\
Denoted by $\{\lambda_{i_1}, \lambda_{i_2}, \cdots, \lambda_{i_s} \}$ be the set of all different eigenvalues of $\mathcal{Q}$ satisfying
$\lambda_{i_1}< \lambda_{i_2}< \cdots < \lambda_{i_s}.$
We shall prove (\ref{zero4}) by induction on $s>1$.
 We can assume that the set of  all different eigenvalues of $\mathcal{Q}$ are $\{\lambda_{l_1}, \cdots,\lambda_{l_{s-k}},\lambda_s -k+1,\cdots,\lambda_s-1, \lambda_s\}$ satisfying $\lambda_{l_j}\neq \lambda_s -k, \forall j=1,2,\cdots s-k.$ So $\lambda_s -k$ is NOT an eigenvalue of $\mathcal{Q}$.\\

Set $V_1:=\oplus_{j=1}^{s-k}V_{\lambda_{i_j}}$ and $V_2:=\oplus_{j=0}^{k-1}V_{\lambda_s-j}.$
Then $V_1$ and $V_2$ are submodules of $V$ and satisfies
$$V=V_1\oplus V_2,$$
$$X \circ Y =0, \forall X \in V_1,  Y \in V_2,$$
and
$$X \circ Y \in V_i, \forall X , Y \in V_i, \forall i=1,2.$$
When $s=2$ then by theorem 1, (\ref{zero4}) holds obviously.\\
Suppose  (\ref{zero4}) holds  when $V=\oplus_{j=1}^{s-1}V_{\lambda_j}$.
Suppose $e_1, e_2,\cdots,e_t $ is a base of $V_1$, and $e_{t+1}, e_{t+2},\cdots,e_m$ is a base of $V_2$. Denoted by $B_{(\beta)}$ the matrix of endomorphism $$e_\beta \circ_{|V_2}:V_2 \longrightarrow V_2 , \forall \beta\in \{t+1, t+2, \cdots, m\}.$$
By straightforward computation, we have
\begin{equation*}
\begin{aligned}
C_{\left( \beta \right)}=\left( \begin{matrix}
	O&		O\\
	O&		B_{\left( \beta \right)}\\
\end{matrix} \right)
\end{aligned}
\end{equation*}
by $C_{(j)}\overline{C_{(j)}^t}-\overline{C_{(j)}^t}C_{(j)}=0,$ we have $B_{(\beta)}\overline{B_{(\beta)}^t}-\overline{B_{(\beta)}^t}B_{(\beta)}=0,\forall \beta\in \{t+1, t+2, \cdots, m\}$.
So the $\mathcal{O}_M$-module $V_2$ together with product $\circ $ satisfying all the assumptions in Lemma \ref{pm1}, by
Lemma \ref{pm1}, we conclude that $e_\alpha \circ e_\beta =0, \forall \alpha, \beta \in \{t+1, t+2, \cdots, m\},$
that is, $B_{(\beta)}=0,\forall \beta\in \{t+1, t+2, \cdots, m\}.$ Hence $$C_{(\beta)}=0, \forall \beta\in \{t+1, t+2, \cdots, m\}.$$

By the induction hypothesis, we obtain $C_{(i)}=0, \forall i\in \{1, 2, \cdots, t\} $
So we conclude  that  $C_{(j)}=0, \forall j\in \{1, 2, \cdots, m\},$ which is equivalent to $\Phi=0.$

Finally, since $\Phi=0.$  when $\mathcal{U}=0$ holds, we get $\mathcal{U}^\dag=0$ and $\Phi^\dag=0$.In this case, the structure connection $$\tilde{D}=D'+\frac{1}{z}\Phi +(\frac{\mathcal{U}}{z}-\mathcal{Q}-z \mathcal{U}^\dag+\frac{w}{2}\Id)\frac{dz}{z}
=D'+(-\Lambda+\frac{w}{2}\Id)\frac{dz}{z},$$ here $\Lambda=diag(\lambda_1,\cdots,\lambda_m).$
\hfill $\qed$
\end{proof}
\begin{proof}
{\it of corollary \ref{rank3}.}
By the assumption $\mathcal{Q}=-\kappa \mathcal{Q}\kappa$, we conclude that the matrices $-\mathbb{Q}$ and $\overline{\mathbb{Q}}$ have the same eigenvalue polynomial. So we can assume that $\{\lambda, 0, -\lambda \}$ is the set of all eigenvalues of $\mathbb{Q}$, $\mathcal{Q}=\mathcal{Q}^\dag$ so $\lambda$ is a real differential function on some coordinate neighborhood. Obviously, $\{\lambda-1, -1, -\lambda-1 \}$  is the set of eigenvalues of $\mathbb{Q}-I_3$.\\
(1) If neither $\pm 1$ nor $\pm \frac{1}{2}$ is an eigenvalue of $\mathbb{Q}$, then $\mathbb{Q}-I_3$ and $\mathbb{Q}$ have no common eigenvalues, then by Theorem 2, we conclude that t $\Phi_{|U_p}=0.$\\
(2) $D'$ is a holomorphic connection, there is a flat holomorphic local frame such that the matrix $\mathbb{Q}$ of $\mathcal{Q}$ is constant,
If $1$  is an eigenvalue of $\mathbb{Q}$, then $-1$ is also an eigenvalue of $\mathbb{Q}$. So we  can choose a a flat holomorphic local frame $X_1, X_2, X_3$ such that $\mathbb{Q}=\diag (1, 0, -1)$ on some neighborhood $U_p$. Similar discussion for $\pm \frac{1}{2}$ is an eigenvalue of  $\mathbb{Q}$.
\hfill $\qed$
\end{proof}
The proof of corollary \ref{rank2} is similar.

\bigskip
In \cite{taka}, A. Takahashi study the extend moduli space of elliptic curves, and prove that it can be equipped with a positive-definite CDV-structure $(M, g, \circ, e,\mathcal{E},\kappa)$. In this structure, the matrix of the endomorphism $\mathcal{Q}=D^{'}_{\mathcal{E}}-\mathcal{L}_{\mathcal{E}}- \frac{2-d}{2}\Id$ is given by $\mathbb{Q}=\diag (\frac{1}{2}, -\frac{1}{2})$ and the Chern connection $D^{'}$ of Hermitian Einstein metric $h$ is holomorphic.

\section{Other Results}\label{}
We get a sufficient condition for  a tuple $\left( \mathcal{T} _{M}^{1,0}\longrightarrow M, D^{'}+\bar{\partial}, \Phi , h,  \mathcal{U} =0, \mathcal{Q} \right)$ be an integrable harmonic Higgs bundle.
\begin{corollary}\label{U3}
Let $\left( M, h, D^{'}+\bar{\partial} \right)$ be a Hermitian manifold, $\rm i.e.$\quad $h$ is positive definite and $D^{'}+\bar{\partial} $ is the Chern connection of $h$, and $D'$ is holomorphic connection. Given any flat holomorphic local frame
$e_1,e_2, \cdots, e_m$ on $U$ satisfying $h(e_i,e_j)=\delta_{ij}$, any constant matrix $\mathbb{Q}$ satisfying $\overline{\mathbb{Q}}=\mathbb{Q}^t$ determined locally a holomorphic endormorphism $\mathcal{Q}$ of the holomorphic tangent bundle.
Define $\Phi_X Y=0,\mathcal{U}=0$.
Then \\
 $\left( \mathcal{T} _{M}, \Phi , h,  \mathcal{U} =0, \mathcal{Q} \right) $ is an integrable harmonic Higgs bundle on U.\\
\end{corollary}

\begin{proof}
{\it  of corollary \ref{U3}.}\\
By lemma\ref{Lemma2}, there is a $D^{'}$-flat holomorphic local frame $e_1,e_2, \cdots,e_m,$ such that
$\mathcal{Q} e_j=\lambda _je_j,\forall j=1,2,...,m$, here $\lambda_j$ are constants.\\

Note that $X\circ Y=-\Phi _XY\,\, ,$ $\because \Phi _XY=\Phi _XY,$
$\therefore X\circ Y=Y\circ X.$

\begin{claim}$1^{\circ}$\label{Tc1}
If $\mathcal{Q} X=\lambda X, \mathcal{Q} Y=\mu Y,  \mu \ne \lambda \,\, $ then $    X\circ Y=0.$
\end{claim}
In fact, $-\left[ \Phi ,\mathcal{Q} \right] +\Phi =0\Longleftrightarrow \left[ -\Phi ,\mathcal{Q} \right] =-\Phi,
$
$$\therefore \,\,\left[ -\Phi _X,\mathcal{Q} \right] \left( Y \right) =\left( -\Phi _X \right) \left( Y \right).
$$
Direct computation shows that
$$\mathcal{Q} \left( -\Phi _XY \right) =\left( \mu -1 \right) \left( -\Phi _XY \right) .$$
\begin{align}\label{t57}
{\rm i.e.} \quad \mathcal{Q} \left( X\circ Y \right) =\left( \mu -1 \right) X\circ Y
\end{align}
Similarly,
$\left[ -\Phi _Y,\mathcal{Q} \right] \left( X \right) =\left( -\Phi _Y \right) \left( X \right) $ implies
\begin{align}\label{t58}
 \mathcal{Q} \left( Y\circ X \right) =\left( \lambda -1 \right) Y\circ X.
\end{align}
Then by (\ref{t57}),(\ref{t58}) and the assumption that $\Phi$ is symmetric, we have
$$X\circ Y=0.$$

\begin{claim}$2^{\circ}$\label{Tc2}  $\forall X, Y,Z\in \Gamma \left( U,\mathcal{T} _{M}^{1,0} \right) $, if $\mathcal{Q} X=\lambda X, \mathcal{Q} Y=\lambda Y, \mathcal{Q} Z=\lambda Z$, then $ X\circ Y\circ Z=0$.
\end{claim}
In fact, by $\left[ -\Phi ,\mathcal{Q} \right] =-\Phi $ we have
$$
\left[ -\Phi _{X\circ Y},\mathcal{Q} \right] \left( Z \right) =-\Phi _{X\circ Y}Z
$$
\begin{align}\label{t61}
\Longleftrightarrow \mathcal{Q} \left( \left( X\circ Y \right) \circ Z \right) =\left( \lambda -1 \right) \left( \left( X\circ Y \right) \circ Z \right)
\end{align}
If $X\circ Y=0$,then $ \left( X\circ Y \right) \circ Z=0.$
If $X\circ Y\ne 0  $, then $$\mathcal{Q} \left( X\circ Y \right) =\left( \lambda -1 \right) \left( X\circ Y \right). $$
$$ \therefore
\left[ -\Phi _Z,\mathcal{Q} \right] \left( X\circ Y \right) =\left( -\Phi _Z \right) \left( X\circ Y \right)
$$
\begin{align}\label{t62}
\Longleftrightarrow \mathcal{Q} \left( Z\circ \left( X\circ Y \right) \right) =\left( \lambda -2 \right) \left( Z\circ \left( X\circ Y \right) \right).
\end{align}
Then by (\ref{t61}),(\ref{t62}) and the assumption that $\Phi$ is symmetric, we have
$$\left( X\circ Y \right) \circ Z=Z\circ \left( X\circ Y \right) =0.$$
By claim $1^{\circ}$ and $2^{\circ}$ we have $\forall X,Y,Z\in \mathcal{T} _{M}^{1,0},$ if $X,Y,Z$
are eigenvectors, then we have
$$\left( X\circ Y \right) \circ Z=0.$$
$$
\therefore \,\,X\circ \left( Y\circ Z \right) =\left( Y\circ Z \right) \circ X=0=\left( X\circ Y \right) \circ Z.
$$
\begin{claim}$3^{\circ}$\label{Tc3} The product $\circ$ satisfies:
$$\forall \xi ,\eta ,\zeta \in \Gamma \left( U,\mathcal{T} _{M}^{1,0} \right) ,   \left( \xi \circ \eta \right) \circ \zeta =0,
$$
therefore This product $\circ$ has the associative law, {\rm i.e.} $\Phi \land \Phi =0.$
\end{claim}
In fact, $\left\{ e_1,e_2,...,e_m \right\} $ is a local frame. By claim $1^{\circ}$ and $2^{\circ}$ ,
$$\forall i,j,k,  \quad
\left( e_i\circ e_j \right) \circ e_k=0.
$$
Let $\xi =f^ie_i, \eta =g^je_j, \zeta =h^ke_k$, then
$$
\left( \xi \circ \eta \right) \circ \zeta =f^ig^jh^k\left( e_i\circ e_j \right) \circ e_k=0.
$$
So we can conclude that $\Phi\wedge \Phi=0$ holds.\\
Since $D^{'}\left( \Phi \right)=0 $ , $D^{'}+\bar{\partial}
$ is compatible with $h$ and $\bar{\partial}\left( \Phi ^{\dagger} \right) =0$ holds,
straight forward computation shows that:
$$
\forall \alpha,\beta, \quad h\left( \left( \bar{\partial}_{\bar{X}_i}\left( \Phi _{\bar{X}_j}^{\dagger} \right) -\bar{\partial}_{\bar{X}_j}\left( \Phi _{\bar{X}_i}^{\dagger} \right) \right) e_{\alpha},e_{\beta} \right) =h\left( e_{\alpha},\left( D_{X_i}^{'}\left( \Phi _{X_j} \right) -D_{X_j}^{'}\left( \Phi _{X_i} \right) \right) e_{\beta} \right),$$
where $X_i=\frac{\partial}{\partial t^i}$, while $t^1,t^2,...,t^m$ are local coordinates.\\
$ D^{'}\left( \Phi \right) =0$  and $h$ is positive definite, $ \therefore \bar{\partial}\left( \Phi ^{\dagger} \right) =0.
$
Since $\mathcal{U} =0,   $ we get $\left[ \Phi ,\mathcal{U} \right] =0,$ and $         \mathcal{U} ^{\dagger}=0.$
By $-\left[ \Phi ,\mathcal{Q} \right] +\Phi =0$,
we get $$D^{'}\left( \mathcal{U} \right) -\left[ \Phi ,\mathcal{Q} \right] +\Phi =0,
$$
$$D^{'}\left( \mathcal{Q} \right) +\left[ \Phi ,\mathcal{U} ^{\dagger} \right] =D^{'}\left( \mathcal{Q} \right) +0=0.
.
$$
$\therefore \left( M, h, \Phi , \mathcal{U} =0, \mathcal{Q} \right) $ is a integrable harmonic Higgs bundle.
\hfill $\qed$
\end{proof}



\bibliographystyle{amsplain}
\bibliography{scoFMwtts}
\end{document}